\documentclass[a4paper]{amsart}

\usepackage{a4wide}
\usepackage{amsmath,amsfonts,amssymb}
\usepackage{amsthm}
\usepackage{ifthen}
\usepackage{array}
\usepackage{url}
\usepackage{multirow}
\usepackage{graphicx}
\usepackage{enumitem}

\usepackage{amssymb}
\usepackage{graphicx}

\usepackage{url}
\usepackage{tabularx}
\usepackage{longtable}

\usepackage{pdflscape}

\RequirePackage[T1]{fontenc}
\RequirePackage[utf8]{inputenc}

\theoremstyle{plain}
\newtheorem{theorem}{Theorem}[section]
\newtheorem{corollary}[theorem]{Corollary}
\newtheorem{proposition}[theorem]{Proposition}

\theoremstyle{definition}

\newtheorem*{remark}{Remark}
\newtheorem*{example}{Example}

\newcommand{\BH}{{\rm BH}}
\newcommand{\CC}{{\rm C}}

\newcommand{\boxtensor}{{\Box\kern-9.03pt\raise1.42pt\hbox{$\times$}}}

\newcommand{\ord}{\mathrm{ord}}

\newcommand{\sE}{{\mathcal E}}

\newcommand{\Q}{{\mathbb Q}}
\newcommand{\R}{{\mathbb R}}

\newcommand{\Z}{{\mathbb Z}}

\newcommand{\be}{\begin{eqnarray}}
\newcommand{\ee}{\end{eqnarray}}

\newcommand{\dd}{\displaystyle}

\newcommand{\q}{\mathfrak q}

\begin{document}

\title[Perfect Sequences]{Non-Existence of Some Nearly Perfect Sequences, Near Butson-Hadamard Matrices, and Near Conference Matrices}

\author[A.\ Winterhof]{Arne Winterhof}
\author[O.\ Yayla]{O\u{g}uz Yayla}
\author[V.\ Ziegler]{Volker Ziegler}

\address[Arne Winterhof, O\u{g}uz Yayla, Volker Ziegler]{Johann Radon Institute for Computational and Applied Mathematics (RICAM)\\
Austrian Academy of Sciences\\
Altenbergerstr. 69\\
4040 Linz\\Austria}
\email{arne.winterhof@ricam.oeaw.ac.at}
\email{oguz.yayla@ricam.oeaw.ac.at}
\email{volker.ziegler@ricam.oeaw.ac.at}

\begin{abstract}
In this paper we study the non-existence problem of (nearly) perfect (almost) $m$-ary sequences via their connection to (near) Butson-Hadamard (BH) matrices and (near) conference matrices. 
Firstly, we apply a result on vanishing sums of roots of unity and a result of Brock on the unsolvability of certain equations over a cyclotomic number field to derive
non-existence results for near BH matrices and near conference matrices. Secondly, we refine the idea of Brock in the case of cyclotomic number fields whose ring of integers is not a principal ideal domains
and get many new non-existence results. 
\end{abstract}
\maketitle

\section{Introduction}

For an integer $m\ge 2$ let $\zeta_m$ denote a primitive complex $m$-th root of unity.
We call a $v$-periodic sequence $\underline{a} = (a_0,a_1,\ldots,a_{v-1},\ldots)$
an {\em $m$-ary sequence} if $a_0,a_1,\ldots,a_{v-1}\in \sE_m=\{1,\zeta_m,\zeta_m^2,\ldots,\zeta_m^{m-1}\}$ 
and an {\em almost $m$-ary sequence} if $a_0=0$ and $a_1,\ldots,a_{v-1}\in \sE_m$.

For $0\leq t \leq v-1$, the 
\textit{autocorrelation 
function} $C_{\underline{a}}(t)$ is defined by
$$ C_{\underline{a}}(t) = \sum_{i=0}^{v-1}{a_i\overline{a_{i+t}}},$$
where $\overline{a}$ is the complex conjugate of $a$. 

An $m$-ary or almost $m$-ary sequence $\underline{a}$ of period $v$ is called a \textit{perfect sequence} (PS) if $C_{\underline{a}}(t)=0$ for all $1\leq t \leq v-1$. 
Similarly, an almost $m$-ary sequence $\underline{a}$ of period $v$ is called a \textit{nearly perfect sequence} (NPS) of type $\gamma\in \{-1,+1\}$ 
if $C_{\underline{a}}(t)=\gamma$ for all $1 \leq t \leq v-1$.
We extend this definition to any $\gamma \in \Z[\zeta_m] \cap \mathbb{R}$ with "small" absolute value with respect to $n$.
%Notice that $\gamma = C_{\underline{a}}(t) = C_{\underline{a}}(n+1-t) = \overline{\gamma} \in \Z[\zeta_m]  \cap \mathbb{R}$. 
%However, $\gamma$ is not necessarily an integer as the example $n=1$, $m=5$, $\underline{a}=(a_0,a_1) = (\zeta_5, {\zeta_5}^{-1})$ and $C_{\underline{a}}(1)= \frac{\pm \sqrt{5}-1}{2}$ shows.
%Note that a NPS of type $\gamma = 0$ refers to a PS.

PS and NPS have several applications such as signal processing and radar, see 
for example \cite{BJL,gogo} and references therein.

PS and NPS can be identified with circulant near Butson-Hadamard matrices and conference matrices.
A square matrix $H$ of order $v$ with entries in $\sE_m$ is called a \textit{near Butson-Hadamard matrix} $\BH_{\gamma}(v,m)$ of type $\gamma$ if 
$H\overline{H}^T= (v-\gamma)I + \gamma J$ for a 
$\gamma \in \R \cap \Z[\zeta_m]$. A $\BH_0(v,m)$ is called {\em Butson-Hadamard matrix} and a $\BH_0(v,2)$ is a {\em Hadamard matrix}.
A square matrix $C$ of order $v$ with $0$ on the diagonal and all off-diagonal entries in $\sE_m$ is called a \textit{near conference matrix} $\CC_{\gamma}(v,m)$ of type $\gamma$ if 
$C\overline{C}^T= (v-1-\gamma)I + \gamma J$ for a 
$\gamma \in \R\cap \Z[\zeta_m]$. A $\CC_0(v,m)$ is called {\em conference matrix}.
A square matrix $H=(h_{ij})$ of order $v$ is called {\em circulant} if $h_{i+1\bmod v,j+1\bmod v} = h_{i,j}$ for all $0\le i,j<v$.

In this paper we prove several new non-existence results on $BH_\gamma(v,m)$ and $\CC_\gamma(v,m)$ which can be interpreted as non-existence results for PS and NPS.
For earlier non-existence results on PS and NPS see \cite{CTZ2010,MaNg2009,OYY2012} and on Butson-Hadamard matrices \cite{Brock,Arne2000}.

First, in Section~\ref{sec.pre} and \ref{sec:self} we derive some new non-existence results on $BH_\gamma(v,m)$ and $\CC_\gamma(v,m)$ using well-known methods.
However, in Section~\ref{sec.mainresult} we present a new idea which works for all $m$ such that $\Z[\zeta_m]$ is not a principal ideal domain (e.g.\ $m=23,29,31,37,39,41,43,46,47,49,\ldots$)
and obtain some non-existence results which could not be obtained by the other methods before. 

In the last section we interpret the results on $BH_\gamma(v,m)$ and $\CC_\gamma(v,m)$ as results on PS and NPS and complete earlier tables for non-existence parameters. 

%In the following we show that $\gamma$ must be an integer if $m \in \{ 2,3,4,6\}$ for the circulant near BH and conference matrices $\BH_{\gamma}(v,m)$ and $\CC_{\gamma}(v,m)$.
%\begin{proposition}
%Let $H = (h_{i,j})$ be a circulant near Butson-Hadamard matrix $\BH_{\gamma}(v,m)$ or circulant near conference matrix $\CC_{\gamma}(v,m)$. If $m \in \{ 2,3,4,6\}$, then $\gamma \in \Z$.
%\end{proposition}
%\begin{proof}
%For $m = 2$ it is trivial since in this case $H$ consists of only integers. For the remaining cases, 
%Notice that for $j \neq k$ 
%$$\gamma = \sum_{j=0}^{v-1}{h_{ij}\overline{h_{jk}}} = \sum_{j=0}^{v-1}{h_{ji}\overline{h_{kj}}} 
%= \sum_{j=0}^{v-1}{\overline{h_{kj}}h_{ji}} = \overline{\sum_{j=0}^{v-1}{h_{kj}\overline{h_{ji}}}} = \overline{\gamma}.$$
%This implies $\gamma \in \Z[\zeta_m] \cap \mathbb{R}=\Z[\zeta_m+\zeta_m^{-1}]$. Therefore, $\gamma \in \Z$ if $\zeta_m+\zeta_m^{-1} \in \Z$, 
%which is true for $m\in \{2,3,4,6\}$.
%that is if \be \label{QQ} \Q(\zeta_m + \zeta_m^{-1}) = \Q.\ee
%On the other hand, the extension degree of $\Q(\zeta_m)$  is 2 over $\Q(\zeta_m + \zeta_m^{-1})$, and $\varphi(m)$ over $\Q$ for any integer $m \geq 3$ and a primitive $m$-th root of unity $\zeta_m$. Therefore, if $\varphi(m) = 2$, that is $m \in \{3,4,6\}$, then (\ref{QQ}) holds.%\end{proof}

\section{A result based on vanishing sums of roots of unity} \label{sec.pre}
In this section we present a non-existence result on $\BH_\gamma(v,m)$ and $\CC_\gamma(v,m)$ based on the following result on vanishing sums of roots of unity due to Lam and Leung \cite{LL2000}.
\begin{proposition}\label{thm.LL} Let $m$ be an integer with prime factorization $m = p_1^{a_1}p_2^{a_2}\ldots p_\ell^{a_\ell}$. If there are  
$m$-th roots of unity $\xi_1,\xi_2,\ldots,\xi_v$ with $\xi_1+\xi_2+\ldots+\xi_v = 0$, then $v = p_1t_1+p_2t_2+\ldots+p_\ell t_\ell$ with non-negative integers $t_1,t_2,\ldots,t_\ell$.
\end{proposition}

\begin{remark} For the case that $m=p^a$ is the power of a prime $p$ see also \cite{Arne2000}.
In this case, if there is a vanishing sum of $v$ $m$-th roots of unity, $v$ must be divisible by $p$ and infinitely many $v$ are excluded. However, if $m$ is 
divisible by at least
two primes $p_1\ne p_2$, any sufficiently large $v$ is of the form $v=t_1p_1+t_2p_2$ with non-negative integers $t_1$ and $t_2$ and Proposition~\ref{thm.LL} excludes only a finite number
of small $v$. In particular, if $m$ is divisible by $6$ only $v=1$ is excluded. 
\end{remark}

An immediate consequence of Proposition \ref{thm.LL} is the following corollary.

\begin{corollary}\label{cor.LLex} Let $m$ be a positive integer with prime factorization  $m = p_1^{a_1}p_2^{a_2}\ldots p_\ell^{a_\ell}$ and $-\gamma$ be a 
sum of $g$ $m$-th roots of unity. 
(Note that $g$ is not unique.)
Then a $\BH_{\gamma}(v,m)$ exists only if $\dd v + g = p_1t_1+p_2t_2+\ldots+p_\ell t_\ell$ for some non-negative integers $t_1,t_2,\ldots,t_\ell$ 
and a $\CC_{\gamma}(v,m)$ exists only if $\dd v-2 + g= p_1t_1+p_2t_2+\ldots+p_\ell t_\ell$ for some non-negative integers $t_1,t_2,\ldots,t_\ell$.  If $\gamma$ 
is an integer, then 
we can choose $\dd g = - \gamma + m\frac{\gamma + \vert \gamma \vert}{2}$.
\end{corollary} 

\begin{proof}
Suppose that $H=(h_{i,j})$ is a $\BH_{\gamma}(v,m)$ (resp.~$\CC_{\gamma}(v,m)$), then
$$
\sum_{i=0}^{v-1}{h_{i,j}\overline{h_{k,i}}}-\gamma=0, \quad k \neq j.
$$
So, by Proposition \ref{thm.LL} we get the first part of the theorem. For the second part let $\gamma$ be an integer. If $\gamma \leq 0$, we can choose $g=-\gamma$. If $\gamma > 0$, then we use $-1 = \sum\limits_{i=1}^{m-1}{\zeta_m^i}$,
where $\zeta_m$ is any primitive $m$-th root of unity, and can choose $g = (m-1)\gamma$. 
\end{proof}

\begin{example}
We performed an exhaustive search on all $4 \times 4$ 5-ary circulant matrices for checking the existence of a circulant $\BH_{\gamma}(4,5)$. We obtain that a circulant $\BH_{\gamma}(4,5)$ exists only for $\gamma \in \{ -1,-\zeta_5^3 - \zeta_5^2 +1, \zeta_5^3+\zeta_5^2+2,4 \}$. 
We note that for these circulant $\BH_{\gamma}(4,5)$ we can choose values of $g$ in $\{ 1,6, 16,16 \}$. Obviously, $v+g$ is always divisible by $5$. 
In addition, we note that $\vert \zeta_5^3 + \zeta_5^2 +2 \vert \approx 0.38$, hence one can obtain a  circulant $\BH_{\gamma}(4,5)$ for a nonzero $\gamma$ such that $\vert \gamma \vert < 1$.
\end{example}

\section{A result based on a self-conjugacy condition} \label{sec:self}In this section we present non-existence results on $\BH_\gamma(v,m)$ and $\CC_\gamma(v,m)$ based on results of Brock \cite{Brock}.  
%We start with the following Lemma. 

%\begin{lemma}\label{lem.nearly}
%Let $M$ be a $v \times v$ matrix of the form $M=(\kappa-\gamma)I+\gamma J$.  
%Then $\det(M)=(\kappa+(v-1)\gamma)(\kappa-\gamma)^{v-1}$.
%\end{lemma}

%\begin{proof}
%The determinant does not change if we first subtract the last row from each of the remaining rows and then add the sum of the first $v-1$ columns to the last 
%column. 
%So, the obtained matrix is a lower triangular matrix with the desired determinant.
%\end{proof}

Let $p$ be a prime and $m$ a positive integer with $\gcd(p,m) = 1$. We say that $p$ is \textit{self-conjugate} modulo $m$ if the order $f$ of $p$ modulo $m$ is even 
and $p^{f/2}\equiv -1\bmod m$.
Corollary~\ref{cor.Brock} below
is based on the fact that for a positive integer $w$ there exists no solution $\alpha$ to the  equation $\alpha \overline{\alpha} 
= w$ over $\Q[\zeta_m]$ if the square-free part of $w$ is divisible by a prime which is self-conjugate modulo $m$, see \cite[Theorem~3.1]{Brock}. 
More precisely, \cite[Theorem~2.5]{Brock} implies:\\
\begin{itemize}\item  If there is a $\BH_\gamma(v,m)$, there is $\alpha \in \Q(\zeta_m)$ with  
$$\alpha \overline{\alpha}=w_1,$$
where
\begin{equation}\label{w_1}w_1=\left\{ \begin{array}{ll} (\gamma+1)v-\gamma, & \mbox{$v$ odd},\\ ((\gamma+1)v-\gamma)(v-\gamma), &\mbox{$v$ even}.\end{array}\right.
\end{equation}
\item If there is a $\CC_\gamma(v,m)$, there is $\alpha\in \Q(\zeta_m)$ with
$$\alpha \overline{\alpha}=w_2,$$
where
\begin{equation}\label{w_2} w_2=\left\{ \begin{array}{ll} (v-1)(\gamma+1), & \mbox{$v$ odd},\\ (v-1)(\gamma+1)(v-1-\gamma), &\mbox{$v$ even}.\end{array}\right.
\end{equation}
\end{itemize}

Now \cite[Theorem~3.1]{Brock} immediately implies the following result.
\begin{corollary}\label{cor.Brock}
  Let $v, m$ be positive integers and $\gamma \in \Z$. 
  Let $v_1$ (resp.\ $v_2$) be the square-free part of $w_1$ (resp.\ $w_2$) defined by $(\ref{w_1})$ (resp.\ $(\ref{w_2})$).
  If there is a prime divisor $p$ of $v_1$ (resp.\ $v_2$) such that
  $p$ is self-conjugate modulo $m$, then there exists no $\BH_{\gamma}(v,m)$ (resp.\ $\CC_\gamma(v,m)$). 
\end{corollary}

\begin{remark}
\begin{enumerate}[label=(\roman{*})]
\item If $\gamma<-1$, there is neither a $\BH_\gamma(v,m)$ nor a $\CC_\gamma(v,m)$ since $w_1$ and $w_2$ are negative but $\alpha\overline{\alpha}$ is non-negative.
\item For $\BH_0(v,m)$ Corollary~\ref{cor.Brock} is \cite[Theorem~4.2 (i)]{Brock}.
\item Let $p$ be a prime and $m=q^r$ or $m=2q^r$ for a prime $q\equiv 3\bmod 4$ with $\gcd(p,q)=1$ (that is $\varphi(m)/2$ is odd, where $\varphi$ denotes Euler's totient function).
If $p$ is a quadratic non-residue modulo $q$, then $p$ is self-conjugate modulo $m$.  
Hence, Corollary~\ref{cor.Brock} generalizes \cite[Theorem~5]{Arne2000} which deals with $\BH_0(v,q^r)$ and $\BH_0(v,2q^r)$ with a prime $q\equiv 3\bmod 4$.
\end{enumerate} 
\end{remark}

\section{Results using prime ideal decomposition over cyclotomic fields}\label{sec.mainresult}

In this section we refine the idea of Brock \cite{Brock}. We use the ideal decomposition of the principal ideal of a rational 
integer over the cyclotomic number field $\Q(\zeta_m)$ to find more values $w$ for which the equation $\alpha\overline\alpha=w$ has no 
solution in $\Z[\zeta_m]$ which excludes the existence of several $\BH_\gamma(v,m)$ and $\CC_\gamma(v,m)$.

More precisely, let $H$ be a $\BH_{\gamma}(v,m)$ for some integer $\gamma$. Then, $H\overline{H}^T=(v-\gamma)I+\gamma J$, $\det(H) \in \Z[\zeta_m]$ and 
$\det(H\overline{H}^T)=\det(H)\overline{\det(H)}=((\gamma+1)v-\gamma)(v-\gamma)^{v-1}$. 
Therefore, we want to find criteria for the unsolvability of the equation 
\begin{equation} \label{eq:ideal}
\alpha \overline{\alpha} = ((\gamma+1)v-\gamma)(v-\gamma)^{v-1}
\end{equation} 
over $\Z[\zeta_m]$. 
Analogously, there is no $\CC_{\gamma}(v,m)$ for some integer $\gamma$ if
\begin{equation} \label{eq:ideal2}
\alpha \overline{\alpha} = (\gamma+1)(v-1)(v-1-\gamma)^{v-1}
\end{equation} 
has no solution $\alpha\in \Z[\zeta_m]$. 

Note that Brock \cite{Brock} studied such equations over $\Q(\zeta_m)$ instead of $\Z[\zeta_m]$, where it was allowed to cancel squares. The following example 
shows that $\alpha \overline{\alpha} = w$ for some positive integer $w$ may have a solution $\alpha \in \Q(\zeta_m)$ even if there is none in $\Z[\zeta_m]$.

\begin{example}
We have $ \alpha \overline{\alpha} = 2$ with $\alpha = \frac{3 + \sqrt{-23}}{4} \in \Q(\sqrt{-23}) \leq \Q(\zeta_{23}).$ 
However, there is no solution $\alpha\in \Z[\zeta_{23}]$ since $(2) = (2,\frac{1 + \sqrt{-23}}{2})(2,\frac{1 - \sqrt{-23}}{2})$ is product of two non-principal prime ideals in $\Z[\zeta_{23}]$.
\end{example}
We now present a simple criterion  for the unsolvability of $\alpha \overline{\alpha} = w$ over $\Z[\zeta_m]$ if $\Z[\zeta_m]$ is not a principal ideal domain. 
We denote by $h_m$ the class number of the cyclotomic number field $\Q(\zeta_m)$ (and assume $h_m\ge 2$ in the following).

\begin{theorem}\label{th:Dioph}
 Let $t\ge 1$ and $e\ge 0$ be rational integers and $q$ be a rational prime such that $q\nmid t$,
\begin{equation}\label{ord} \ord_{m}(q) = \varphi(m)/2.
\end{equation} 
and $\gcd(2e+1-2k,h_m) = 1$ for all 
integers $0\leq k \leq e-1$.
 Then the equation
 \begin{equation} \label{eq:Dioph-Main} \alpha \overline{\alpha} = tq^{2e+1} \end{equation}
 has no solution over $\Z[\zeta_m]$, provided that every prime ideal $\mathfrak t \vartriangleleft \Z[\zeta_m]$ 
 with $\mathfrak t|(t)$ is principal and every prime ideal $\q \vartriangleleft\Z[\zeta_m]$ with $\q|(q)$ is non-principal.
\end{theorem}

\begin{proof}
Assume that there exists a solution $\alpha$ to \eqref{eq:Dioph-Main}. We consider the prime ideal decomposition over $\Q(\zeta_m)$ of the right hand side of 
\eqref{eq:Dioph-Main}. According to \cite[Theorem~2.13]{Was1997} we have $(q)=\mathfrak q_1 \dots \mathfrak q_g$, where 
$g=\varphi(m)/\ord_{m}(q)$ and $g=2$ by $(\ref{ord})$. 

We may assume that 
\begin{equation}\label{eq:alpha_decomp}
(\alpha)=\mathfrak t \mathfrak q_1^{2e+1-k} \mathfrak q_2^{k}=\mathfrak t \mathfrak q_1^{2e+1-2k} (\mathfrak q_1\mathfrak q_2)^{k},
\end{equation}
where $\mathfrak t|(t)$ is a principal ideal and $k$ is some integer with $0\leq k \leq e-1$.
Since $(\mathfrak q_1\mathfrak q_2)=(q)$ is principal, the right hand side of \eqref{eq:alpha_decomp} is principal only if 
$\mathfrak q_1^{2e+1-2k}$ is principal. Then, the order of $\mathfrak q_1$ is a divisor of $\gcd(2e+1-2k,h_m) = 1$, contradicting that $\mathfrak q_1$ is 
non-principal. 
\end{proof}

\begin{example}
\begin{itemize}
\item $\alpha\overline{\alpha}=73$ has no solution $\alpha\in \Z[\zeta_{23}]$ since $(73)=(73,(1+\sqrt{-73})/2)(73,(1-\sqrt{-73})/2)$ over $\Z[\zeta_{23}]$ 
and $ord_{23}(73)=11=\varphi(23)/2$.
\item $\alpha \overline{\alpha} = 3^3$ has no solution $\alpha\in \Z[\zeta_{47}]$ 
since  $(3) = (3,(1 + \sqrt{-47})/2)(3,(1 - \sqrt{-47})/2)$ is a product of two non-principal prime ideals in $\Z[\zeta_{47}]$
and $h_{47}=695$ is not divisible by $3$.
\end{itemize}
\end{example}

Applying Theorem~\ref{th:Dioph} to \eqref{eq:ideal}, we get a criterion for the non-existence of $\BH_{\gamma}(v,m)$:

\begin{corollary} \label{thm.main.had}
Let $v,m$ be positive integers and $\gamma \geq -1$ be an integer such that $((\gamma+1)v-\gamma)(v-\gamma)^{v-1} = tq^{2e+1}$ with integers $t\ge 1$ and $e\ge 0$ such 
that $\gcd(2e+1-2k,h_m) = 1$ for all integers $0\leq k \leq e-1$, a prime $q$ with $\gcd(t,q)=1$ and 
$\ord_m(q) = \varphi(m)/2$.  Let every prime ideal $\mathfrak t \vartriangleleft \Z[\zeta_m]$ 
 with $\mathfrak t|(t)$ be principal and every prime ideal $\q \vartriangleleft\Z[\zeta_m]$ with $\q|(q)$ be non-principal. Then there exists no 
$\BH_{\gamma}(v,m)$. 
\end{corollary}

\begin{remark} \label{rem.main}
\begin{enumerate}[label=(\roman{*})]
\item If $m =p^f$ is a prime power and $\ord_m(q)$ is even, then $q$ is self conjugate modulo $m$ and a $\BH(v,m)$ does not exist due to 
Brock \cite{Brock} (cf.\ Corollary~\ref{cor.Brock}). Thus in the prime power case we consider only $\BH_\gamma(v,m)$  such that 
$\ord_m(q)=\varphi(m)/2$ is odd, i.e.\ $p \equiv 3 \mod 4$. 
\item We know by Corollary~\ref{cor.LLex} that if a $\BH_{\gamma}(v,p^f)$ exists, then $p \mid v-\gamma$, and if a $\BH_{\gamma}(v,2p^f)$ exists then $v 
-\gamma = 2t_1 +pt_2$ for some non-negative integers $t_1$ and $t_2$.\end{enumerate}
\end{remark}

%\begin{landscape}
\begin{table}[!t]
\caption{The class number $h_m$ of $\Q(\zeta_m)$ for $m \leq 70$ \cite{Was1997}.}
\label{tab.classnbr}
\small{
\begin{center}
\begin{tabular}{|cc|cc|cc|cc|cc|cc|cc|}
\hline
$m$ &$h_m$ &$m$ &$h_m$ &$m$ &$h_m$ &$m$ &$h_m$ &$m$ &$h_m$ &$m$ &$h_m$ &$m$ &$h_m$ \\\hline
1 &1 &11 &1 &21 &1 &31 &9 &41 &121 &51 &5 &61 &76301 \\
\hline
2 &1 &12 &1 &22 &1 &32 &1 &42 &1 &52 &3 &62 &9 \\
\hline
3 &1 &13 &1 &23 &3 &33 &1 &43 &211 &53 &48891 &63 &7 \\
\hline
4 &1 &14 &1 &24 &1 &34 &1 &44 &1 &54 &1 &64 &17 \\
\hline
5 &1 &15 &1 &25 &1 &35 &1 &45 &1 &55 &10 &65 &64 \\
\hline
6 &1 &16 &1 &26 &1 &36 &1 &46 &3 &56 &2 &66 &1 \\
\hline
7 &1 &17 &1 &27 &1 &37 &37 &47 &695 &57 &9 &67 &853513 \\
\hline
8 &1 &18 &1 &28 &1 &38 &1 &48 &1 &58 &8 &68 &8 \\\hline
9 &1 &19 &1 &29 &8 &39 &2 &49 &43 &59 &41421 &69 &69 \\
\hline
10 &1 &20 &1 &30 &1 &40 &1 &50 &1 &60 &1 &70 &1 \\\hline
\end{tabular}
\end{center}
}
\end{table}
%\end{landscape}

\begin{example}
$\BH_2(25,23), \BH_2(117,23), \BH_2(163,23), \BH_1(32,31)$  do not exist by Corollary~\ref{thm.main.had}. The existence of $\BH_{2}(v,46)$ for $v \in 
\{$25, 51, 55, 109, 111, 117, 163, 183, 193, 201, 213, 225, 247, 315, 363, 385, 433, 477$\} $ such that $v \leq 500$ are excluded by Corollary~\ref{thm.main.had}. Similarly, $\BH_{2}(v,72)$ for $v \in \{$35, 141, 181, 183, 221, 231, 243, 245, 251, 395, 431, 475$\} $ such that $v \leq 500$ are excluded 
by Corollary~\ref{thm.main.had}. These matrices are the smallest examples whose existence cannot be excluded by Corollary~\ref{cor.LLex} or Corollary~\ref{cor.Brock}. 
Finally, the existence of $\BH_2(115,94)$  can be excluded by Corollary~\ref{thm.main.had}, but not by 
Corollary~\ref{cor.LLex} or Corollary~\ref{cor.Brock}.
\end{example}

Similarly to Corollary~\ref{thm.main.had}, applying Theorem~\ref{th:Dioph} to \eqref{eq:ideal2}, we get a criterion for the non-existence of $\CC_{\gamma}(v,m)$:

\begin{corollary} \label{thm.main.conf}
Let $v,m$ be positive integers and $\gamma \geq -1$ be an integer such that $(\gamma+1)(v-1)(v-1-\gamma)^{v-1} = tq^{2e+1}$ with integers $t\ge 1$ and $e\ge 0$ such that $\gcd(2e+1-2k,h_m) = 1$
for all integers $0\leq k \leq e-1$, a prime $q$ with $\gcd(t,q)=1$ and 
$\ord_m(q) = \varphi(m)/2$.  Let every prime ideal $\mathfrak t \vartriangleleft \Z[\zeta_m]$ 
 with $\mathfrak t|(t)$ be principal and every prime ideal $\q \vartriangleleft\Z[\zeta_m]$ with $\q|(q)$ be non-principal.Then there exists no $\CC_{\gamma}(v,m)$. 
\end{corollary}

\begin{remark}
\begin{enumerate}[label=(\roman{*})]
\item In the cases $\gamma=0$ and $h_m|v$ or $2|v$, or $\gamma=-1$, Corollary~\ref{thm.main.conf} does not yield a non-existence result for near conference matrices. 
\item We know by Corollary~\ref{cor.LLex} that if a $\CC_{\gamma}(v,p^f)$ exists, then $p \mid v-2-\gamma$, and if $\CC_{\gamma}(v,2p^f)$ exists, then $v-2-\gamma = 2t_1 +pt_2$ 
for some nonnegative integers $t_1$ and $t_2$.
\end{enumerate}
\end{remark}

\begin{example}
$\CC_2(3362,23), \CC_2(26,46),  \CC_2(124,72)$ do not exist by using Corollary~\ref{thm.main.conf}. These matrices are the smallest examples whose existence 
 cannot be excluded by Corollary~\ref{cor.LLex} or Corollary~\ref{cor.Brock}.
\end{example}

\section{Application to sequences} \label{sec.results}

In this section we apply the results of the previous sections on matrices to sequences. 
We note that a NPS of type $\gamma = 0$ refers to a PS.

We start with the new results on the non-existence of $m$-ary NPS by using Corollary~\ref{cor.Brock}. We list in Table~\ref{tab.m_ary} all 
periods $v \leq  100$ such that no $m$-ary NPS exists for $m \leq 100$. We list only the new non-existence results. The pairs $(v,m)$ listed in \cite{MaNg2009} such that an $m$-ary NPS of period $v$ does not exist are not listed again in Table 
\ref{tab.m_ary}. In other words, we extend the tables in 
\cite{MaNg2009} for composite values of $m \leq 100$. 

\begin{table}[!b]
\caption{Periods $v \leq 100$ such that an $m$-ary NPS does not exist.}\label{tab.m_ary}
\begin{center}
\small{
\begin{tabular}{|l|p{0.7\textwidth}|}
\hline
 &$(v,m)$\\
\hline
$\gamma=0$
&(45,9),
(75,25),
(99,9)\\
\hline
$\gamma=-1$
&(44,9),
(74,25),
(98,9)\\
\hline
$\gamma=1$
&(11,10),
(17,4),
(26,25),
(28,9),
(28,27),
(29,4),
(29,14),
(29,28),
(35,34),
(43,6),
(43,14),
(43,21),
(43,42),
(46,9),
(53,4),
(53,26),
(59,58),
(65,4),
(67,22),
(71,10),
(71,14),
(73,6),
(73,9),
(73,18),
(76,25),
(81,4),
(81,8),
(81,10),
(83,82),
(89,4),
(93,46),
(94,93),
(100,9)\\
\hline
\end{tabular}
}
\end{center}
\end{table}

We note that a $21$-ary PS of period 105 does not exist by using Corollary~\ref{cor.Brock}. It differs from the other pairs in the first  row of 
Table~\ref{tab.m_ary} as $21$ is product of distinct primes.
One of the smallest undecided cases is whether a $6$-ary NPS of period $11, 12, 13$ for $\gamma = -1,0,1$, respectively, exists.

\begin{table}[!t]
\caption{Periods $v \leq 100$ such that an almost $m$-ary NPS does not exist.}
\label{tab.almost_m_ary}
\small{
\begin{center}
\begin{tabular}{|l|p{0.7\textwidth}|}
\hline
 &$(v,m)$\\
\hline
$\gamma=0$
&(11,9),
(27,25),
(35,33),
(47,9),
(59,57),
(67,65),
(71,69),
(77,25),
(79,77),
(83,9),
(83,27),
(83,81),\\
\hline
$\gamma=1$
&(6,3),
(7,2),
(7,4),
(8,5),
(11,2),
(12,3),
(12,9),
(13,2),
(13,5),
(13,10),
(14,11),
(15,2),
(15,4),
(16,13),
(18,3),
(18,5),
(20,17),
(21,2),
(21,3),
(21,6),
(21,9),
(21,18),
(22,19),
(23,2),
(23,4),
(24,3),
(25,2),
(27,2),
(28,5),
(28,25),
(29,2),
(29,13),
(29,26),
(30,3),
(30,9),
(30,27),
(31,2),
(31,4),
(31,7),
(31,14),
(31,28),
(32,29),
(34,31),
(35,2),
(36,3),
(36,11),
(36,33),
(37,17),
(38,5),
(39,2),
(39,4),
(40,37),
(41,2),
(42,3),
(42,13),
(43,2),
(43,4),
(43,5),
(43,8),
(43,10),
(44,41),
(45,2),
(45,3),
(45,6),
(46,43),
(47,2),
(47,4),
(48,3),
(48,5),
(48,9),
(49,2),
(52,7),
(52,49),
(53,2),
(53,5),
(53,10),
(53,25),
(53,50),
(54,3),
(54,17),
(55,2),
(55,4),
(56,53),
(57,2),
(58,5),
(58,11),
(58,55),
(59,2),
(60,3),
(60,19),
(60,57),
(61,2),
(61,29),
(61,58),
(62,59),
(63,2),
(63,4),
(64,61),
(66,3),
(66,7),
(66,9),
(66,21),
(66,63),
(67,2),
(67,4),
(68,5),
(68,13),
(68,65),
(69,2),
(69,3),
(69,6),
(69,11),
(69,22),
(69,33),
(69,66),
(70,67),
(71,2),
(71,4),
(71,17),
(71,34),
(72,3),
(72,23),
(72,69),
(75,2),
(76,73),
(77,2),
(77,37),
(77,74),
(78,3),
(78,5),
(78,25),
(79,2),
(79,4),
(79,19),
(79,38),
(79,76),
(80,7),
(80,11),
(80,77),
(81,2),
(81,3),
(81,6),
(81,13),
(81,26),
(83,2),
(84,3),
(84,9),
(84,27),
(84,81),
(85,2),
(85,41),
(85,82),
(86,83),
(87,2),
(87,4),
(88,5),
(88,17),
(89,2),
(90,3),
(90,29),
(91,2),
(92,89),
(93,2),
(93,3),
(93,5),
(93,6),
(93,9),
(93,10),
(93,18),
(94,7),
(94,13),
(95,2),
(95,4),
(96,3),
(97,2),
(98,5),
(98,19),
(100,97)\\
\hline
\end{tabular}
\end{center}
}
\end{table}

Similarly, we present non-existence results of almost $m$-ary NPS by using Corollary~\ref{cor.Brock}. We list in Table \ref{tab.almost_m_ary} all periods $v \leq 100$ such that an almost $m$-ary NPS does not exist for $m \leq100$. We only list new non-existence results. The pairs $(v,m)$ listed 
in \cite{CTZ2010} such that an almost $m$-ary NPS of period $v$ does not exist are not listed again in Table 
\ref{tab.almost_m_ary}. In particular, we extend the tables in \cite{CTZ2010} to composite $m \leq 100$. 
We note that Corollary~\ref{cor.Brock} does not 
yield any non-existence result for $\gamma = -1$ as the determinant vanishes.
Finally, let us note that 
one of the smallest open cases is whether an almost $6$-ary NPS of period $13,14,15$ for $\gamma = -1,0,1$ respectively, exists.

In addition, by using Corollary~\ref{thm.main.had}, we conclude that a 23-ary NPS of periods 25, 117 and 163 for $\gamma = 2$ and a 31-ary NPS of period 32 for $\gamma = 1$ do not exist. Similarly, by using Corollary~\ref{thm.main.conf}, no almost 23-ary NPS 
of period 3362 for $\gamma = 2$ exists.

\section*{Acknowledgment}
The first author is supported by the Austrian Science Fund (FWF): Project F5511-N26 
which is part of the Special Research Program "Quasi-Monte Carlo Methods: Theory and Applications".
The second author is supported by T\"{U}B\.{I}TAK under Grant No.\ 2219. The last author is supported by the Austrian Science Fund (FWF) under the project P24801-N26. We also want to thank Clemens Heuberger for several 
fruitful discussions on this topic during the special semester on
Applications of Algebra and Number Theory held at the Johann Radon Institute for Computational and Applied Mathematics (RICAM) in 
Linz, October 14 - December 13, 2013.

\bibliographystyle{amsplain}
\bibliography{wyz_hadamard}

\end{document}